\documentclass{article}
\usepackage{amsfonts}
\usepackage{graphicx}
\usepackage{amsmath,amsthm,latexsym,amssymb,amstext}
\newtheorem{thm}{Theorem}[section]
\newtheorem{def1}[thm]{Definition}

\newtheorem{lemma}[thm]{Lemma}
\newtheorem{corollary}[thm]{Corollary}
\newtheorem{rem}[thm]{Remark}
\newtheorem{prop}[thm]{Proposition}
\newtheorem{example}[thm]{Example}

\usepackage[pdfstartview=FitH,
bookmarksnumbered=true,bookmarksopen=true,%
colorlinks=true,pdfborder=001,citecolor=blue,%
linkcolor=blue,urlcolor=blue]{hyperref}

\begin{document}
\title{\Large \bf A Complex Version of G-Expectation and Its Application to Conformal Martingale}

\author{Huilin Zhang \thanks{%
School of Mathematics, Shandong University,
huilinzhang2014@gmail.com} } \maketitle
\date{}

\begin{abstract}
This paper is concerned with the connection between G-Brownian
Motion and analytic functions. We introduce the complex version of
sublinear expectation, and then do the stochastic analysis in this
framework. Furthermore, the conformal G-Brownian Motion is
introduced together with a representation, and the corresponding
conformal invariance is shown.

\end{abstract}



\date{}
\maketitle

\newcommand{\EC}{\hat{\mathbb{E}}_\mathbb{C}}
\newcommand{\HC}{\mathcal {H}_\mathbb{C}}
\newcommand{\OHEC}{(\Omega,\,\HC,\,\EC)}
\newcommand{\E}{\hat{\mathbb{E}}}
\newcommand{\R}{\mathbb{R}}
\newcommand{\OHE}{(\Omega,\,\mathcal{H},\,\E)}
\newcommand{\C}{\mathbb{C}}

\section{Introduction}
G-expectation was first introduced by peng to handle risk measures,
stochastic volatility and model
uncertainty(see\cite{P07a},\cite{P08a}). The idea of G-expectation
lies in the new explanation of independence and distribution. The
main difference between G-expectation and linear expectation is, of
course, sublinearity, which is mathematically explained by a
powerful tool in PDE, viscosity solution theory. A typical
G-expectation is constructed on linear space of random variables
$L_G^1(\Omega),$ the completion of $Lip(\Omega)$ of Lipschitz
cylinder functions on d-dimensional continuous path space
$\Omega=C_0^d[0,\infty ).$ The corresponding G-Brownian Motion on
this space is the canonical process, i.e. $B_t(\omega)=\omega(t),
\omega \in \Omega.$ A sublinear expectation $\E$ on $(\Omega,
Lip(\Omega))$ is defined as a viscosity solution of a heat equation,
which is determined by a sublinear monotone function G. Then the
random variable space $Lip(\Omega)$ is extended to completion by the
norm $\E[|\cdot|],$ and we denote the completion as $L_G^1(\Omega).$
By such construction, a time consistent conditional expectation is
introduced naturally on $L_G^1(\Omega).$\\

G-expectation theory is thriving in recent
years.(see\cite{DHP11},\cite{P10},\cite{Song11}) In classical case,
we know there are many interesting connection between Brownian
Motion and analytic functions. A fundamental one is Paul
$L\acute{e}vy's$ theorem, that is, if $f(z)$ is analytic nonconstant
function, and $B_t$ is a two dimensional Brownian Motion, the
process $f(B_t),t \geq 0$ is again Brownian Motion, probably moving
at a variable speed.(see \cite{BD79}) $L\acute{e}vy's$ result leads
to the possibility to study analytic functions probabilistically. On
the other hand, a beauty of linear expectation is that a
distribution can be uniquely described by its characteristic
function. However, there
is little work on G-expectation under complex framework.\\

In this paper, we give the stochastic analysis of G-expectation in
complex case, and furthermore, the conformal G-Brownian Motion and
conformal martingales. We find that a special G-Brownian Motion,
namely the conformal G-Brownian Motion still holds conformal
invariance, i.e. it is still conformal after a transformation by a
nonconstant analytic function. We should notice that conformal
G-Brownian Motion is much more complicated than the classical
case.\\

This paper is organized as follows. In section 2, we recall some
basic results of G-Brownian. In section 3 and 4, we define the
sublinear expectation under complex framework and the complex
G-normal distribution and G-Brownian Motion. In the fifth part, we
do the stochastic analysis in this framework and give It$\hat{o}'s$
formula. In the last part, we introduce the conformal martingale and
prove conformal invariance.

\section{Preliminaries}

We start from some basic notes and results of G-expectation theory.
More details can be read from \cite{P07a}, \cite{P08a}, \cite{P10}.

\subsection{Sublinear Expectation and G-Expectation}

Let $\Omega$ be a given set and $\mathcal {H}$ be a linear space of
real valued functions on $\Omega$ containing constants. Furthermore,
suppose $|X| \in \mathcal{H}$ if $X \in \mathcal{H}.$ The space
$\Omega$ is viewed as sample space and $\mathcal {H}$ is the space
of random variables.

\begin{def1}
A sublinear expectation $\E$ is a functional $\E: \mathcal{H}
\rightarrow \R$ satisfying

$(i)$ constant preserving:
$$
\E(c)=c,\qquad \forall\;c\in\,\mathbb{R}$$

$(ii)$ positive homogeneity:
$$\E(\lambda X)=\lambda\E(X),\qquad \lambda\geq0 \quad X\in
\mathcal{H}$$

$(iii)$ constant transferability:
$$\E(X+c)=\E(X)+c, \qquad c\in \mathbb{R} \quad X\in
\mathcal{H}$$

$(iv)$ monotonicity:
$$\E(X_1) \geq \E(X_2) \qquad if\; X_1\geq X_2$$

The triple $(\Omega, \mathcal{H}, \E)$ is called a sublinear
expectation space.
\end{def1}

\begin{def1}
Let $X_{1}$ and $X_{2}$ be two $n$-dimensional random vectors
defined respectively in sublinear expectation spaces $(\Omega_{1}
,\mathcal{H}_{1},\E_{1})$ and $(\Omega_{2},\mathcal{H}
_{2},\E_{2})$. They are called identically distributed, denoted by
$X_{1}\overset{d}{=}X_{2}$, if $\E_{1}[\varphi(X_{1}
)]=\E_{2}[\varphi(X_{2})]$, for all$\  \varphi \in C_{b.Lip}
(\mathbb{R}^{n})$.
\end{def1}

\begin{def1}
In a sublinear expectation space $(\Omega,\mathcal{H} ,\E)$, a
random vector $Y=(Y_{1},\cdot \cdot \cdot,Y_{n})$, $Y_{i}\in
\mathcal{H}$, is said to be independent of another random vector
$X=(X_{1},\cdot \cdot \cdot,X_{m})$, $X_{i}\in \mathcal{H}$ under
$\E[\cdot]$, denoted by $Y\bot X$, if for every test function
$\varphi \in C_{b.Lip}(\mathbb{R}^{m}\times \mathbb{R}^{n})$ we have
$\E[\varphi(X,Y)]=\E[ \E [\varphi(x,Y)]_{x=X}]$.
\end{def1}

\begin{def1}(maximal distribution)
A d-dimensional random vector $\eta=(\eta_1,\cdots ,\eta_d)$ on a
sublinear expectation space $(\Omega,\mathcal {H},\E)$ is called
maximal distributed if $\eta$ satisfies:
\[
a\eta+b\bar{\eta}\overset{d}{=}(a+b)\eta, \; a,b \geq 0,
\]
where $\bar{\eta}$ is an independent copy of $\eta$, i.e.,
$\bar{\eta}\overset{d}{=}\eta$ and $\bar{\eta}\bot \eta.$

\end{def1}

\begin{rem}\label{md}
By the definition of maximal distribution, we can get that there
exists a bounded, closed and convex subset $\Gamma \subset \R^d$
such that for any $\varphi \in C_{b.lip}(\R^d),$ we have:
$$
\E[\varphi(\eta)]=\max _{y \in \Gamma} \varphi(y).
$$
When $d=1,$ we have $\Gamma=[\underline{\mu},\bar{\mu}],$ where
$\bar{\mu}=\E[\eta]$ and $\underline{\mu}=-\E[-\eta]$. Furthermore,
in this case we denote $\eta
\overset{d}{=}N(\{\underline{\mu},\bar{\mu}\}\times{0}).$

In the classical case, the maximal distributed random is a constant.
\end{rem}

\begin{def1}($G$-normal distribution)
A $d$-dimensional random vector $X=(X_{1},\cdot \cdot \cdot,X_{d})$
in a sublinear expectation space $(\Omega,\mathcal{H},\E)$ is called
$G$-normally distributed if $\mathbb{\hat{E}}[|X|^{3}]<\infty$ and
for each $a,b\geq0$
\[
aX+b\bar{X}\overset{d}{=}\sqrt{a^{2}+b^{2}}X,
\]
where $\bar{X}$ is an independent copy of $X$, and
\[
G(A):=\frac{1}{2}\E[\langle AX,X\rangle]:\mathbb{S}%
_{d}\rightarrow \mathbb{R},
\]
Here $\mathbb{S}_{d}$ denotes the collection of $d\times d$
symmetric matrices.
\end{def1}

By\cite{P10}we know that $X=(X_1,\cdots, X_d )$ is G-normal
distributed iff $u(t,x):=\E[\varphi(x+\sqrt{t}X)],$ $(t,x) \in
[0,\infty)\times \R^d,$ $\varphi \in C_{b.Lip}(\R^d),$ is the
viscosity solution of the following G-heat equation:
$$
\partial_t u-G(D_x^2u)=0, \ u(0,x)=\varphi(x).
$$
The function $G(\cdot):\mathbb{S} _{d} \rightarrow \mathbb{R}$ is a
monotonic, sublinear functional on $\mathbb{S} _{d},$ from which we
can deduce that there exists a bounded, convex and closed subset
$\Sigma \subset \mathbb{S}_{d}^{+}$ such that
\[
G(A)=\frac{1}{2}\sup_{B\in \Sigma}\mathrm{tr}[AB],
\]
where $\mathbb{S}_{d}^{+}$ denotes the collection of nonnegative
matrixes in $\mathbb{S}_{d}$.

Here is the typical construction of G-expectation
from\cite{P08a}.For simplicity, we only consider the one dimensional
condition.

Let $\Omega=C_{0}(\mathbb{R}^{+})$ be the space of all
$\mathbb{R}$-valued continuous paths $(\omega_{t})_{t\in
\mathbb{R}^{+}}$,
with $\omega_{0}=0$, equipped with the distance%
\[
\rho(\omega^{1},\omega^{2}):=\sum_{i=1}^{\infty}2^{-i}[(\max_{t\in
\lbrack 0,i]}|\omega_{t}^{1}-\omega_{t}^{2}|)\wedge1].
\]
and the canonical process is defined by $B_{t}(\omega)=\omega_{t}$,
$t\in [0,\infty)$, for $\omega \in \Omega$. Then we define
\[
L_{ip}(\Omega):=\{ \varphi(B_{t_{1}},...,B_{t_{n}}):n\geq1,t_{1},...,t_{n}%
\in [0,\infty),\varphi \in C_{b.Lip}(\mathbb{R}^{ n})\}.
\]

Give a monotonic and sublinear function $G:\mathbb{S}_{d}\rightarrow
\mathbb{R}.$ The corresponding G-expectation $ \hat{\E}$ is a
sublinear expectation on $L_{ip}(\Omega)$ satisfying:

\[
\mathbb{\hat{E}}[X]=\mathbb{\tilde{E}}[\varphi(\sqrt{t_{1}-t_{0}}\xi_{1}%
,\cdot \cdot \cdot,\sqrt{t_{m}-t_{m-1}}\xi_{m})],
\]
for all $X=\varphi(B_{t_{1}}-B_{t_{0}},B_{t_{2}}-B_{t_{1}},\cdot
\cdot \cdot,B_{t_{m}}-B_{t_{m-1}})$ with $0\leq
t_{0}<t_{1}<\cdots<t_{m}<\infty$, where $\xi_{1},\cdot \cdot
\cdot,\xi_{n}$ are identically distributed $d$-dimensional
$G$-normally distributed random vectors in a sublinear expectation
space $(\tilde{\Omega},\tilde{\mathcal{H}},\mathbb{\tilde{E}})$ such
that $\xi_{i+1}$ is independent of $(\xi_{1},\cdot \cdot
\cdot,\xi_{i})$ for every $i=1,\cdot \cdot \cdot,m-1$.

\begin{def1}
Let $\Omega_{t}=\{ \omega_{\cdot \wedge t}:\omega \in \Omega \}$ for
$t\geq0$. For each
$\xi=\varphi(B_{t_{1}}-B_{t_{0}},B_{t_{2}}-B_{t_{1}},\cdot \cdot
\cdot,B_{t_{m}}-B_{t_{m-1}})$, the conditional $G$-expectation of
$\xi$ under $\Omega_{t_{i}}$ is defined by
\[
\mathbb{\hat{E}}_{t_{i}}[\varphi(B_{t_{1}}-B_{t_{0}},B_{t_{2}}-B_{t_{1}}%
,\cdot \cdot \cdot,B_{t_{m}}-B_{t_{m-1}})]
\]%
\[
=\tilde{\varphi}(B_{t_{1}}-B_{t_{0}},B_{t_{2}}-B_{t_{1}},\cdot \cdot
\cdot,B_{t_{i}}-B_{t_{i-1}}),
\]
where
\[
\tilde{\varphi}(x_{1},\cdot \cdot
\cdot,x_{i})=\mathbb{\hat{E}}[\varphi
(x_{1},\cdot \cdot \cdot,x_{i},B_{t_{i+1}}-B_{t_{i}},\cdot \cdot \cdot,B_{t_{m}%
}-B_{t_{m-1}})].
\]

\end{def1}

For each fixed $T\geq0$, we define
\[
L_{ip}(\Omega_{T}):=\{ \varphi(B_{t_{1}},...,B_{t_{n}}):n\geq1,t_{1}%
,...,t_{n}\in \lbrack0,T],\varphi \in C_{b.Lip}(\mathbb{R}^{n})\}.
\]
It is simple that $L_{ip}(\Omega_{T_{1}})\subset
L_{ip}(\Omega_{T_{2}})\subset L_{ip}(\Omega)$ for $T_{1}<T_{2}$.
Furthermore, we denote $L_{G}^{p}(\Omega)$ and
$L_{G}^{p}(\Omega_{T})$, $p\geq1$ as the completion of
$L_{ip}(\Omega)$ and $L_{ip}(\Omega_{T})$ under the norm $\Vert \xi
\Vert_{p}=(\mathbb{\hat{E}} [|\xi|^{p}])^{1/p}$. Consequently,
$L_{G}^{p_{1}}(\Omega)\subset L_{G}^{p_{2}}(\Omega)$ for $p_{1}\geq
p_{2}\geq1$.

By the above construction we can get that the $G$-expectation
$\mathbb{\hat{E}}[\cdot]$ can be continuously extended to a
sublinear expectation on $(\Omega,L_{G}^{1}(\Omega))$ and it is
still denoted by $\mathbb{\hat{E}}[\cdot]$. For each given $t\geq0$,
the conditional $G$-expectation
$\mathbb{\hat{E}}_{t}[\cdot]:L_{ip}(\Omega)\rightarrow
L_{ip}(\Omega_{t})$ can be continuously extended as a mapping $\mathbb{\hat{E}}%
_{t}[\cdot]:L_{G}^{1}(\Omega)\rightarrow L_{G}^{1}(\Omega_{t})$ and
satisfies the following properties:

\begin{description}
\item[(i)] If $X$, $Y\in L_{G}^{1}(\Omega)$, $X\geq Y$, then $\mathbb{\hat{E}%
}_{t}[X]\geq \mathbb{\hat{E}}_{t}[Y]$;

\item[(ii)] If $X\in L_{G}^{1}(\Omega_{t})$, $Y\in L_{G}^{1}(\Omega)$, then
$\mathbb{\hat{E}}_{t}[X+Y]=X+\mathbb{\hat{E}}_{t}[Y]$;

\item[(iii)] If $X$, $Y\in L_{G}^{1}(\Omega)$, then $\mathbb{\hat{E}}%
_{t}[X+Y]\leq \mathbb{\hat{E}}_{t}[X]+\mathbb{\hat{E}}_{t}[Y]$;

\item[(iv)] If $X\in L_{G}^{1}(\Omega_{t})$ is bounded, $Y\in L_{G}^{1}%
(\Omega)$, then $\mathbb{\hat{E}}_{t}[XY]=X^{+}\mathbb{\hat{E}}_{t}%
[Y]+X^{-}\mathbb{\hat{E}}_{t}[-Y]$;

\item[(v)] If $X\in L_{G}^{1}(\Omega)$, then $\mathbb{\hat{E}}_{s}%
[\mathbb{\hat{E}}_{t}[X]]=\mathbb{\hat{E}}_{s\wedge t}[X]$, in
particular,
$\mathbb{\hat{E}}[\mathbb{\hat{E}}_{t}[X]]=\mathbb{\hat{E}}[X]$.
\end{description}

\subsection{Stochastic calculus under G-Expectation}

For a partition $\Pi=\{t_0,t_1,\cdots,t_N\}$ of $[0,T]$, we define
$$
\eta_t(\omega):=\Sigma_{k=0}^{N-1}\xi_k(\omega)I_{[t_k,t_k+1)}(t)
$$
where $\xi_k \in L_G^p(\Omega_{t_k})$, $k=0,1,2,\cdots,N-1$ are
given. We denote these processes by $M_G^{p,0}(0,T)$.

\begin{def1}
For any $\eta \in M_G^{p,0}(0,T)$, with the form
$\eta_t(\omega)=\Sigma_{k=0}^{N-1}\xi_k(\omega)I_{[t_k,t_k+1)}(t)$,
the related Bochner integral on $[0,T]$ is defined :
$$
\int_0^T
\eta_t(\omega)dt:=\sum_{k=0}^{N-1}\xi_k(\omega)(t_{k+1}-t_k)
$$

\end{def1}

We then complete the space $M_G^{p,0}(0,T)$ under norm\\
$\|\cdot\|_M^p:=\{{\E}[\int_0^T|\cdot|^p dt]\}^{\frac1p}$ and denote
the completion as $M_G^p(0,T)$. It is clear that $M_G^p(0,T)\subset
M_G^q(0,T),$if $1 \leq q \leq p.$ Here is the definition of
It$\hat{o}'s$ integral.

\begin{def1}
For a $\eta \in M_G^{2,0}(0,T)$ of the form
$$
\eta_t(\omega)=\Sigma_{k=0}^{N-1}\xi_k(\omega)I_{[t_k,t_k+1)}(t)
$$
We define the $it\hat{o}$ integral as the following operator
$I(\cdot) \ : \ M_G^{2,0}(0,T) \mapsto L_G^2(\Omega_T)$ as
following:
$$
I(\eta)=\int_0^{T} \eta_t dB_t:=\sum_{j=0}^{N-1}
\xi_j(B_{t_{j+1}}-B_{t_j}),
$$
where $B_t$ is a G-Brownian Motion, and
$G(\alpha)=\frac12(\bar{\sigma}^2 \alpha^+ -\underline{\sigma}^2
\alpha^-),$ $0 \leq \underline{\sigma}^2 \leq \bar{\sigma}^2 \leq
\infty$
\end{def1}

\begin{lemma}
For the mapping $I:M_G^{2,0}(0,T) \mapsto L_G^2(\Omega_T)$, we have:

\begin{eqnarray}
{\E}[\int_0^T \eta_t dB_t]   & = &0,\\
{\E}[|\int_0^T \eta_t dB_t|^2]& \leq & \bar{\sigma}^2
{\E}[\int_0^T|\eta_t|^2 dt],
\end{eqnarray}

\end{lemma}

Thus we can continuously extend $I$ to a mapping from $M_G^2(0,T)$
to $L_G^2(\Omega_T),$ which is also denoted as $ I$.

\begin{def1}
For a $\eta \in  M_G^{2}(0,T)$, the stochastic integral is defined
as
\begin{eqnarray*}
\int_0^T\eta_t dB_t&:=& I(\eta),\\
 \int_s^t \eta_u dB_u&:=& \int_0^T I_{[s,t]}(u)
\eta_u dB_u,
\end{eqnarray*}

\end{def1}

Here is some basic properties of the integral of G-Brownian motion,
and the proof is omitted.

\begin{prop}
For $\eta, \zeta \in M_G^{2}(0,T)$, and $0\leq s\leq r \leq
t\leq T$, we have:\\
$(i)$ $\int_s^t \eta_u dB_u=\int_s^r \eta_u dB_u+\int_r^t \eta_u
dB_u$.\\
$(ii)$ $\int_s^t(\alpha \eta_u+\zeta_u)dB_u=\alpha \int_s^t \eta_u
dB_u+\int_s^t \zeta_u dB_u$, where $\alpha$ is bounded in
$L_G^1(\Omega_s)$.\\
$(iii)$ $\E[X+\int_r^T \eta_u dB_u|\Omega_s]=\E[X|\Omega_s]$, for $X
\in L_G^1(\Omega_T)$

\end{prop}

Now we consider the quadratic variation process of G-Brownian motion
$(B_t)_{t \geq 0}$ with $B_1\overset{d}{=}N(\{0\} \times
[\underline{\sigma}^2,\bar{\sigma}^2]).$

For a G-B.M. $B_t$ and a partition $\Pi_N$ of $[0,t]$: $0=t_0\leq
t_1\leq,\cdots, \leq t_N=t$, notice
$$
B_t^2=\sum_{j=0}^{N-1}
2B_{t_j^{N}}(B_{t_{j+1}^{N}}-B_{t_j^{N}})+\sum_{j=0}^{N-1}(B_{t_{j+1}^{N}}-B_{t_j^{N}})^2
.$$ As $||\Pi_N||\rightarrow 0$, we can show that
$$
\sum_{j=0}^{N-1}2B_{t_j^{N}}(B_{t_{j+1}^{N}}-B_{t_j^{N}})
\stackrel{L_G^2(\Omega)}{\rightarrow} 2\int_0^t B_u dB_u,
$$
so by the completeness of $L_G^2(\Omega)$,
$\sum_{j=0}^{N-1}(B_{t_{j+1}^{N}}-B_{t_j^{N}})^2$ must also converge
in $L_G^2(\Omega)$.

\begin{def1}
By the argument above, we define
\begin{eqnarray*}
\langle B \rangle_t&:=&\lim_{||\Pi_N||\rightarrow
0}\sum_{j=0}^{N-1}(B_{t_{j+1}^{N}}-B_{t_j^{N}})^2\\
     &=& B_t^2-2\int_0^t B_r dB_r,
     \end{eqnarray*}
and call $\langle B \rangle$ the quadratic variation process of
G-Brownian Motion.

\end{def1}

%
%

We now define the integral of a process $\eta \in M_G^1(0,T)$ with
respect to $\langle B \rangle.$ Firstly, we define a mapping:
$$
Q_{0,T}(\eta)=\int_0^T \eta_t d \langle B
\rangle_t:=\sum_{j=0}^{N-1} \xi_j(\langle B
\rangle_{t_{j+1}}-\langle B \rangle_{t_j}):M_G^{1,0}(0,T)\rightarrow
L_G^1(\Omega_T).
$$
Furthermore, we have the following lemma, which allows as to extend
this mapping to $M_G^1(0,T).$

\begin{lemma}
For each $\eta \in M_G^{1,0}(0,T),$
$$
\E[|\int_0^T \eta_t d\langle B \rangle_t|] \leq \bar{\sigma}^2
\E[\int_0^T |\eta_t|dt].
$$
\end{lemma}

Here two properties w.r.t $\langle B \rangle$ and its integral which
would be used in the next part. The proof can be found in
\cite{P10}.

\begin{prop}
$(i)$ for $\eta \in M_G^2(0,T),$ we have
$$
\E[(\int_0^T \eta_t dB_t)^2]=\E[\int_0^T \eta_t^2 d\langle B
\rangle_t].
$$

$(ii)$ for each fixed $s,t \geq 0,$ $\langle B \rangle_{t+s}-\langle
B \rangle_s$ is identically distributed with $\langle B \rangle_t$
and independent from $\Omega_s.$ In addition, $\langle B \rangle_t$
is
$N([\underline{\sigma}^2t,\bar{\sigma}^2t]\times{0})-distributed.$

\end{prop}

For the multi-dimensional case. Suppose $(B_t)_{t \geq 0}$ be a
d-dimensional G-Brownian motion. For each $a \in \R^d,$ we can
easily check that $(B_t^a)_{t \geq 0}:=(<B_t,a>)_{t \geq 0} $ is a
1-dimensional $G_a$-Brownian Motion, with $G_a(\alpha)=\frac12
(\sigma^2_{aa^T}\alpha^+ - \sigma^2_{-aa^T}\alpha^-),$ where
$\sigma^2_{aa^T}=2G(aa^T)$ and $\sigma^2_{-aa^T}=-2G(-aa^T).$ Then
the related integrals with respect to $B_t^a$ and $\langle B^a
\rangle_t$ are same as the one dimensional case. Furthermore, for
any $a,\bar{a} \in \R^d, $ we define the mutual variation process by
\begin{eqnarray*}
\langle B^a, B^{\bar{a}} \rangle_t &:=& \frac 14 [\langle
B^a+B^{\bar{a}}\rangle_t-\langle B^a-B^{\bar{a}}\rangle_t]\\
                                   &=& \frac 14 [\langle B^{a+\bar{a}}\rangle_t -\langle B^{a-\bar{a}}\rangle_t
                                   ].
                                   \end{eqnarray*}
We can prove that for a sequence of partition
$\pi_t^n,n=1,2,\cdots,$ of $[0,t]$ with $||\pi_t^n||\rightarrow 0,$
we have
$$
\lim_{n\rightarrow \infty}
\sum_{k=0}^{n-1}(B^a_{t_{k+1}^n}-B^a_{t_{k}^n})
(B^{\bar{a}}_{t_{k+1}^n}-B^{\bar{a}}_{t_{k}^n})=\langle
B^a,B^{\bar{a}} \rangle_t.
$$

Now we give the famous $It\hat{o}'s$ formula in G-framework. More
recent progresses can be read from \cite{L-P}.

\begin{thm}{(G-$It\hat{o}'s$ Formula)}
Let $\Phi$ be a twice continuous function on $\R^n$ with polynomial
growth growth for the first and second order derivatives. X is a
$It\hat{o}$ process, i.e.
$$
X_t^\nu=X_0^\nu+\int_0^t \alpha_s^\nu ds +\int_0^t \eta_s^{\nu ij}
d\langle B^i,B^j \rangle_+ \int_0^t \beta _s^{\nu j }dB_s^j
$$
where $\nu=1,...,n,\ i,j=1,...,d,$ $\alpha_s^\nu,\eta_s^{\nu
ij},\beta _s^{\nu j }$ are bounded processes in $M_G^2(0,T).$ Here
repeated indices means summation over the same indices. Then for
each $t \geq s \geq 0$ we have in $L_G^2(\Omega_t):$

\begin{eqnarray*}
\Phi(X_t)-\Phi(X_s)&=&\int_s^t \partial_{x^{\nu}} \Phi(X_u) \beta
_u^{\nu j } dB_u^j + \int_s^t \partial_{x^{\nu}} \Phi(X_u)
\alpha_u^{\nu} du\\
                   &+&\int_s^t[\partial_{x^{\nu}}\Phi(X_u)\eta_u^{\nu
ij}+\frac12 \partial^2_{x^{\mu} x^{\nu}}\Phi(X_u) \beta _u^{\mu i
}\beta _u^{\nu j }]d \langle B^i, B^j \rangle_u
\end{eqnarray*}
\end{thm}

\section{Complex Sublinear Expectation}

Here we try to define a sublinear expectation under complex case,
which means we have to decide what's "sublinear" here. From here on,
when we compare two complex numbers, we compare them as real
vectors. What we do here is to connect the G-expectation with
complex analysis, so we define complex sublinear expectation in the
following heuristic way.

Given a set $\Omega$ and a linear space $\mathcal {H}$ consisting of
real valued functions on $\Omega$, we suppose $\mathcal {H}$ is a
vector lattice and consider the following set of functions on
$\Omega$:
$$\mathcal {H}_\mathbb{C}=\{X+iY|X,Y\in\mathcal {H} \},\,where\,i^2=-1$$

\begin{def1}\label{1}
Suppose $(\Omega\ \mathcal {H}\ \E)$is a sublinear space, we define
a function $\EC$ from $\mathcal{H}_\mathbb{C}$ to complex field as:
$$
\begin{array}{crcl}
\EC:&\mathcal
{H}_\mathbb{C}&\mapsto&\mathbb{C}\\
\\
  &X+iY&\mapsto&\E[X]+i\E[Y]
 \end{array}$$
\end{def1}

\begin{rem}
Notice that for any $Z\in\mathcal {H}_\mathbb{C}$, $Z$ has a unique
expression as $X+iY$, so $\EC$ in well defined.

Here are some basic properties of $\EC$:\\

$(i)$ constant preserving:
$$
\EC(z)=z,\qquad \forall\;z\in\,\mathbb{C}$$

$(ii)$ positive
homogeneity:
$$\EC(\lambda Z)=\lambda\EC(Z),\qquad \lambda\geq0 \quad Z\in
\mathcal{H}_\mathbb{C}$$

$(iii)$ constant transferability:
$$\EC(Z+c)=\EC(Z)+c, \qquad c\in \mathbb{C} \quad Z\in
\mathcal{H}_\mathbb{C}$$

$(iv)$ monotonicity:
$$\EC(Z_1) \geq \EC(Z_2) \qquad if\; Z_1\geq Z_2$$

$(v)$ convexity:
$$\EC[\alpha Z_1+(1-\alpha)Z_2] \leq \alpha
\EC(Z_1)+(1-\alpha)\EC(Z_2), \qquad \alpha \in [0,1]$$

\end{rem}

\begin{rem}
Conversely, given some properties of the complex expectation, we can
define the sublinear expectation in a more general version:
$$ \EC\,:\,\HC \rightarrow \mathbb{C}$$ satisfying $$ \EC[X] \in
\mathbb{R} \quad for \, any \,X\in \mathcal{H}$$ and

\begin{eqnarray*}
(i)&\EC(c)=c,\;\forall c\in\mathbb{C}\\
(ii)&\EC(Z_1)\geq \EC(Z_2),\,Z_1\geq Z_2,\,Z_1,\ Z_2\in \HC\\
(iii)&\EC(Z_1+Z_2)\leq \EC(Z_1)+\EC(Z_2)\\
(iv)&\EC(\lambda Z)=\lambda\EC(Z),\,\lambda \geq 0\,\lambda \in \mathbb{R}\\
(v)&\EC(Z)\leq \EC(X)+i\EC(Y),\;Z=X+iY
\end{eqnarray*}

If $(v)$ in the above is changed to an equation, we get the same
sublinear expectation as definition (\ref{1}) by the representation
theorem in the real case.

\end{rem}

\begin{rem}
For each $\varphi \in C_{l.lip}(\mathbb{C})$,
we suppose
$\varphi(x+iy)=\varphi_1(x,y)+i\varphi_2(x,y)$.
Then we have
 $\varphi_i(x,y)\in
C_{l.lip}(\mathbb{R}^2) $. Also, if $\varphi_i(x,y)\in
C_{l.lip}(\mathbb{R}^2),$ $ i=1,2$, we have
$\varphi(x+iy)(:=\varphi_1(x,y)+i\varphi_2(x,y))$ belongs to
$C_{l.lip}(\mathbb{C})$.

\end{rem}

\begin{def1}
Suppose $Z=(Z_1,\cdots,Z_n)$ a n-dimensional vector on the complex
sublinear space $(\Omega,\;\mathcal {H}_\mathbb{C},\;\EC)$. The
functional from $C_{l.lip}(\mathbb{C}^n)$ to $\mathbb{C}$ defined by
\begin{eqnarray}
 \mathbb{F}_Z(\varphi)&:=&\EC[\varphi(Z)] \nonumber\\
                       &=&\mathbb{E}[\varphi_1(X,Y)]+\mathbb{E}[\varphi_2(X,Y)],
\end{eqnarray}
where $\varphi \in C_{l.lip}(\mathbb{C}^n)$ and
$\varphi(x+iy):=\varphi_1(x,y)+i\varphi_2(x,y)$, is called the
distribution of $Z$ under $\EC$. Also, the triple
$(\mathbb{C}^n,\,C_{l.lip}(\mathbb{C}^n),\,\mathbb{F}_Z)$ forms a
complex sublinear space.

\end{def1}

\begin{def1}
Two n-dimensional random vectors $Z_1$ and $Z_2$ defined on two
complex sublinear expectation spaces $(\Omega_1,\,\mathcal
{H}_\mathbb{C}^1,\,\EC^1)$ and $(\Omega_2,\,\mathcal
{H}_\mathbb{C}^2,\,\EC^2)$ respectively, are called identically
distributed if
$$ \EC^1[\varphi(Z_1)]=\EC^2[\varphi(Z_2)], \qquad  for
\,any \,\varphi \in C_{l.lip}(\mathbb{C}^n).$$ Such relation is
denoted by $Z_1\stackrel{d}{=}Z_2.$

\end{def1}

\begin{prop}\label{10}
Let $\OHEC$ be a complex sublinear expectation. $Z_1,\,Z_2$ are two
random variables and $Z_2$ satisfies
$$ \EC[Z_2]=-\EC[-Z_2] $$
Then we have
$$\EC[Z_1+cZ_2]=\EC[Z_1]+c\ \EC[Z_2],\quad for\,any\,c\in
\mathbb{C}.$$

\end{prop}
\begin{proof}
Suppose $Z_1=X_1+iY_1,\; Z_2=X_2+iY_2,\;c=x+iy$, where $X_i,\, Y_i
\in \mathcal {H},\;and \ x,\ y \in \mathbb{R}$. Since
$\EC[Z_2]=-\EC[-Z_2]$, we have
$$ \mathbb{E}[X_2]=-\mathbb{E}[-X_2],\quad
\mathbb{E}[Y_2]=-\mathbb{E}[-Y_2] $$

This leads to $$ \E[X+aX_2]=\E[X]+a\ \E[X_2],for \, a\in \R,X\in
\mathcal {H} $$ and similar result for $Y_2$.
\begin{eqnarray*}
\EC[Z_1+cZ_2]&=&\EC[(X_1+xX_2-yY_2)+i(Y_1+xY_2+yX_2)]\\
             &=&\E[X_1]+x\E[X_2]-y\E[Y_2]+i\E[Y_1]+ix\E[Y_2]+iy\E[X_2]\\
             &=&\EC[Z_1]+c\EC[Z_2]
\end{eqnarray*}

\end{proof}

Another important definition in sublinear expectation is the
independence. Since our main purpose is to connect real sublinear
expectation with complex analysis, we find the following definition
does this job.

\begin{def1}
In a complex sublinear space $\OHEC$, a random vector $Z_2\in \HC^n$
is said to be independent from another one $Z_1\in \HC^m$ if for
each test function $\varphi \in C_{l.lip}(\mathbb{C}^{m+n})$ we have
$$\EC[\varphi(Z_1,Z_2)]=\EC[\EC[\varphi(z,Z_2)]_{z=Z_1}]$$

\end{def1}

Here is an important lemma to connect the complex case with the real
one.

\begin{lemma}\label{indep}
Let $Z_1=X_1+iY_1,\ Z_2=X_2+iY_2$, where $X_1,\ Y_1 \in \mathcal
{H}^n,\, X_2,\ Y_2 \in \mathcal {H}^m$. Then $Z_2$ is independent of
$Z_1$ if and only if $(X_2,\ Y_2)$ is independent of $(X_1,\ Y_1)$
under $\E$.

\end{lemma}

\begin{proof}
If $Z_2$ is independent of $Z_1$, we need to show that for any
$\varphi \in C_{l.lip}(\mathbb{R}^{2n+2m})$ we have
$$ \E[\varphi(X_1, Y_1, X_2,
Y_2)]=\E[\E[\varphi(x,y,X_2,Y_2)]_{(x,y)=(X_1,Y_1)}]$$

For any fixed $ \varphi \in C_{l.lip}(\mathbb{R}^{2n+2m})$, we take
$\psi := \varphi+i\cdot 0=\varphi$, which belongs to
$C_{l.lip}(\mathbb{C}^{n+m})$.

By the definition of independence under complex case, we have
\begin{eqnarray*}
\E[\varphi(X_1,Y_1,X_2,Y_2)]&=& \EC[\psi(Z_1,Z_2)]\\
                            &=& \EC[\EC[\psi(z,Z_2)]_{z=Z_1}]\\
                            &=& \EC[\E[\varphi(x,y,X_2,Y_2)]_{(x,y)=(X_1,Y_1)}]\\
                            &=& \E[\E[\varphi(x,y,X_2,Y_2)]_{(x,y)=(X_1,Y_1)}]
\end{eqnarray*}

For the backward implication, if $\varphi\in
C_{l.lip}(\mathbb{C}^{m+n})$, we suppose
$\varphi=\varphi_1+i\varphi_2,\ where \varphi_1,\ \varphi_2\in
C_{l.lip}(\mathbb{R}^{2n+2m})$

\begin{eqnarray*}
\EC[\varphi(Z_1,Z_2)]&=&\E[\varphi_1(X_1,Y_1,X_2,Y_2)]+i\E[\varphi_2(X_1,Y_1,X_2,Y_2)]\\
                     &=&\E[\E[\varphi_1(x,y,X_2,Y_2)]_{(x,y)=(X_1,Y_1)}]+i\E[\E[\varphi_2(x,y,X_2,Y_2)]_{(x,y)=(X_1,Y_1)}]\\
                     &=&\EC[\E[\varphi_1(x,y,X_2,Y_2)]_{(x,y)=(X_1,Y_1)}+i\E[\varphi_2(x,y,X_2,Y_2)]_{(x,y)=(X_1,Y_1)}]\\
                     &=&\EC[\EC[\varphi(z,Z_2)]_{z=Z_1}]
\end{eqnarray*}

\end{proof}

%

\begin{rem}{(Completion of Complex Sublinear Expectation Space)}
For a real sublinear expectation space $\OHE$, we denote
$(\Omega,{\mathcal {H}}^p,\E)$ the completion of the real sublinear
expectation space under norm $(\E[|\cdot|^p])^{1/p}$. Then we define
${\mathcal{H}}_{\C}^p=\{X+iY \mid X,Y\in \hat{\mathcal
{H}}^p \}$, and it is also a banach space under norm\\
  $(\E[|\cdot|^p])^{1/p}$ by the completeness of $\hat{\mathcal
{H}}^p$.

\end{rem}

\section{Complex G-normal distribution and G-Brownian Motion}

Here we define the complex G-normal distribution and G-Brownian
Motion.

\begin{def1}{(Complex G-normal distribution)}
A d-dimensional random variable $Z=(Z_1,\cdots,Z_d)$ on a complex
sublinear expectation space $\OHEC$ is called complex G-normal
distributed if
$$ aZ+b\hat{Z}\stackrel{d}{=}\sqrt{a^2+b^2}Z, for \ a,b\geq 0$$
where $\hat{Z}$ is an independent copy of $Z$.

\end{def1}

\begin{example}
Suppose $Z$ is a complex G-normal distributed variable. By the
definition of the complex G-normal distribution, $e^{i\theta}Z$ is
also normal distributed, which means a change of the argument of a
complex G-normal distributed variable is also G-normal distributed.
\end{example}

\begin{rem}
By the definition of identically distributed and lemma(\ref{indep}),
$Z=X+iY$ is complex G-normal distributed if and only if $(X,Y)$ is
G-normal distributed in $\OHE$
\end{rem}

\begin{rem}
If we define $\omega(t,z):=\EC[\varphi(z+\sqrt{t}Z)]$, where
$\varphi \in C_{b.lip}(\C^d),Z\in \HC^d$ and $Z$ is complex G-normal
distributed, we have
\begin{eqnarray*}
\omega(t+s,z)&=&\EC[\varphi(z+\sqrt{t+s}Z)]\\
             &=&\EC[\varphi(z+\sqrt{s}Z+\sqrt{t}\hat{Z})]\\
             &=&\EC[\EC[\varphi(z+\sqrt{s}c+\sqrt{t}Z]_{c=Z}]\\
             &=&\EC[\omega(t,z+\sqrt{s}Z)]
\end{eqnarray*}
If $\omega(t,z)$ is continuously differentiable on t, complex
differentiable on z, and has at most polynomially growth at
infinity, thanks to proposition (\ref{10}), we could get a $G_{\C}$
heat equation:
$$\partial_t\omega(t,z)-G_{\C}(D_2\omega(t,z))=0$$
where $G_{\C}(c)=\frac{1}{2}\EC[cZ^2]$. Here we should notice that
$G$, which has 4 parameters, could determine $G_{\C}$, a 3-parameter
function and the converse is not usually true. However, there are
conditions under which G and $G_{\C}$ are mutually determined, such
as $Z=X+i\hat{X}$, where $\hat{X}$ is the independent copy of the
$\tilde{G}$ normal distributed variable $X$.

\end{rem}

Now we can turn to stochastic analysis. Let's start with complex
G-Brownian Motion.

%

\begin{def1}{complex G-Brownian Motion}
A d-dimensional process $(B_t)_{t\geq 0}$ is called a complex
G-Brownian Motion if the following are satisfied:\\
$(i)$ $B_0(\omega)=0$\\
$(ii)$ For each $s,t\geq0$, the increment $B_{t+s}-B_t$ is complex
G-normal distributed and is independent from
$(B_{t_1},B_{t_2},\cdots,B_{t_n})$, for any $n\in\mathbb{N}$ and
$0\leq t_1\leq t_2\leq \cdots \leq t_n\leq t$.

\end{def1}

%

\section{Stochastic Integral and Related Stochastic Calculus}


\begin{def1}
For each $T\in [0,\infty)$, we set Lipschitz cylinder functions as:
$$
Lip_{\C}(\Omega_T):=\{ \varphi(B_{t_1\wedge T},B_{t_2\wedge
T},\cdots,B_{t_n\wedge T}):n\in \mathbb{N},\ t_1,\cdots,t_n\in
[0,\infty),\ \varphi\in C_{l.lip}(\C^{d\times n}) \}
$$
where $(B_t)_{t\geq 0}$ is a complex G-Brownian Motion. It is clear
that $Lip_{\C}(\Omega_t)\subseteq Lip_{\C}(\Omega_T)$ for $t\leq T$.
Also, we let
$Lip_{\C}(\Omega):=\bigcup_{n=1}^{\infty}Lip_{\C}(\Omega_n)$. Then
we have a complex sublinear expectation called complex G-expectation
as:
$$
\EC[\cdot]\ :\ Lip_{\C}(\Omega)\longmapsto \C.
$$

\end{def1}

\begin{rem}
The corresponding conditional expectation also explains itself as
the real case.
\end{rem}

To complete $Lip_{\C}(\Omega_T)$, we need a small lemma to define a
norm on it.

\begin{lemma}\label{26}
$\OHE$ is a real sublinear space. Then $\| \cdot
\|:=(\E|\cdot|^p)^{\frac1p}$,$p>1$, defines a norm on $\OHEC$.

\end{lemma}

\begin{proof}
%
The proof is similar as the classical one, so we omit it.

\end{proof}

\begin{rem}
We define a norm $\E[|\cdot|^p]^{\frac1p}$ on $Lip_{\C}(\Omega)$ and
denote $L_G^p(\Omega)_{\C}$ as the completion of linear space
$Lip_{\C}(\Omega)$. Also, $L_G^p(\Omega_T)_{\C}$ explains itself.
Then we have the following.

\end{rem}

\begin{prop}
The completion $L_G^p(\Omega_T)_{\C}$ of $Lip_{\C}(\Omega_T)$ under
norm $\E[|\cdot|^p]^{\frac1p}$ have the following expression:
$$
L_G^p(\Omega_T)_{\C}=\{\xi_1+i\xi_2|\ \xi_1,\xi_2\in
L_G^p(\Omega_T)\}
$$
under norm $\E[|\cdot|^p]^{\frac1p}$, where $L_G^p(\Omega_T)$ is the
completion of
$$
Lip(\Omega_T):=\{\varphi(B_{t_1\wedge T}^{(1)},B_{t_2\wedge
T}^{(1)},\cdots,B_{t_n\wedge T}^{(1)},B_{t_1\wedge
T}^{(2)},\cdots,B_{t_n\wedge T}^{(2)})|n \in \mathbb{N},\varphi \in
C_{l.lip}(\R^{d\times 2n})\}
$$
and $B_t=B_t^{(1)}+iB_t^{(2)}$ is the complex G-Brownian Motion.

\end{prop}

\begin{proof}
For any $\xi_1,\xi_2 \in Lip(\Omega_T)$, without loss of generality,
we can suppose they have the same form as
$\xi_i=\varphi_i(B_{t_1\wedge T}^{(1)},B_{t_2\wedge
T}^{(1)},\cdots,B_{t_n\wedge T}^{(1)},B_{t_1\wedge
T}^{(2)},\cdots,B_{t_n\wedge T}^{(2)})$, where $\varphi_i \in
C_{l.lip}(\R^{d\times 2n}),i=1,2$. Then we can take
$\varphi=\varphi_1+i\varphi_2$ and conclude that $\xi_1+i\xi_2 \in
L_G^p(\Omega_T)_{\C}$. The forward inclusion follows from facts
$(i)$ $Lip(\Omega_T)$ is dense in $L_G^p(\Omega_T)$ and $(ii)$
$L_G^p(\Omega_T)_{\C}$ is complete.\\

For the converse inclusion, notice that for any $\varphi \in
C_{l.lip}(\C^{d\times n})$, $\varphi$ has a unique decomposition $
\varphi=\varphi_1+i\varphi_2 $ where $\varphi_1,\varphi_2 \in
C_{l.lip}(\R^{2d \times n})$. It means $Lip_{\C}(\Omega_T) \subseteq
\{\xi_1+i\xi_2|\xi_1,\xi_2 \in L_G^p(\Omega_T)\}$ by the
completeness of $ \{\xi_1+i\xi_2|\xi_1,\xi_2 \in L_G^p(\Omega_T)\}.$
\end{proof}

Now we can define $it\hat{o}'s$ integral by starting from simple
processes as the following: for a partition
$\Pi=\{t_0,t_1,\cdots,t_N\}$ of $[0,T]$, we define
$$
\eta_t(\omega):=\Sigma_{k=0}^{N-1}\xi_k(\omega)I_{[t_k,t_{k+1})}(t)
$$
where $\xi_k \in L_G^p(\Omega_{t_k})_{\C}$, $k=0,1,2,\cdots,N-1$ are
given. We denote these processes by $M_G^{p,0}(0,T)_{\C}$.

\begin{def1}
For any $\eta \in M_G^{p,0}(0,T)_{\C}$, with the form
$\eta_t(\omega)=\Sigma_{k=0}^{N-1}\xi_k(\omega)I_{[t_k,t_{k+1})}(t)$,
the related Bochner integral on $[0,T]$ is defined as natural as:
$$
\int_0^T
\eta_t(\omega)dt:=\sum_{k=0}^{N-1}\xi_k(\omega)(t_{k+1}-t_k)
$$

\end{def1}

Then we can complete the space $M_G^{p,0}(0,T)_{\C}$ under norm\\
$\|\cdot\|_M^p:=\{\E[\int_0^T|\cdot|^p dt]\}^{\frac1p}$ and denote
the completion as $M_G^p(0,T)_{\C}$. Similar as
$L_G^p(\Omega_T)_{\C}$, $M_G^p(0,T)_{\C}$ has the expression:
$$
M_G^p(0,T)_{\C}=\{ \eta_t^{(1)}+i\eta_t^{(2)}\mid
\eta_t^{(1)},\eta_t^{(2)} \in M_G^p(0,T) \}
$$
where $M_G^p(0,T)$ is the corresponding space in real case.

\begin{def1}
For a $\eta \in M_G^{2,0}(0,T)_{\C}$ of the form
$$
\eta_t(\omega)=\Sigma_{k=0}^{N-1}\xi_k(\omega)I_{[t_k,t_k+1)}(t)
$$
We define the $it\hat{o}$ integral as the following operator
$I(\cdot) \ : \ M_G^{2,0}(0,T)_{\C} \mapsto L_G^2(\Omega_T)_{\C}$ as
following:
$$
I(\eta)=\int_0^{T} \eta_t dB_t:=\sum_{j=0}^{N-1}
\xi_j(B_{t_{j+1}}-B_{t_j}),
$$
where $B_t=B_t^{(1)}+iB_t^{(2)}$ is a Brownian Motion.
\end{def1}

\begin{prop}
For the mapping $I:M_G^{2,0}(0,T)_{\C} \mapsto
L_G^2(\Omega_T)_{\C}$, we have:

\begin{equation}\label{2}
 \EC[\int_0^T \eta_t dB_t]    = 0,
\end{equation}
\begin{equation}\label{3}
 \E[|\int_0^T \eta_t dB_t|^2] \leq
16K\E[\int_0^T|\eta_t|^2 dt],
\end{equation}

where $K=(\bar{\sigma}_1^2+\bar{\sigma}_2^2)$,
$\bar{\sigma}_1^2:=\E[(B_1^{(1)})^2]$,
$\bar{\sigma}_2^2:=\E[(B_1^{(2)})^2]$.

\end{prop}

\begin{proof}

We denote
$$
\triangle B_{j}=B_{t_{j+1}}-B_{t_j},
$$
$$
\triangle B_{j}^{(i)}=B_{t_{j+1}}^{(i)}-B_{t_j}^{(i)},i=1,2
$$
For $(\ref{2})$, we notice $ \EC[\int_0^T \eta_t dB_t] =
\EC[\sum_{j=0}^{N-1}\xi_j \triangle B_{j}]$, and

\begin{eqnarray*}
\EC[\xi_j \triangle B_{j}] &=&
\EC[(\xi_j^{(1)}+i\xi_j^{(2)})(\triangle B_{j}^{(1)}+i\triangle
B_{j}^{(2)})]\\
                             &=& \E[\xi_j^{(1)}\triangle B_{j}^{(1)}-\xi_j^{(2)}\triangle
B_{j}^{(2)}]+i\E[\xi_j^{(1)} \triangle B_{j}^{(2)} + \xi_j^{(2)}
\triangle B_{j}^{(1)}].
\end{eqnarray*}
Also, we have $\E[\xi_j^{(i)}\triangle B_{j}^{(k)}]=0,i,k=1,2.$ By
proposition (\ref{10}), these equations imply the conclusion.

For $(\ref{3})$, firstly notice
\begin{eqnarray*}
\E[|\int_0^T\eta_t dB_t|^2]
&=&
\E[|\sum_{j=0}^{N-1}[(\xi_j^{(1)}\triangle
B_j^{(1)}-\xi_j^{(2)}\triangle B_j^{(2)})+i(\xi_j^{(1)}\triangle
B_j^{(2)} +\xi_j^{(2)}\triangle B_j^{(1)})]|^2]\\
&\leq &
\E[(\sum_{j=0}^{N-1}(\xi_j^{(1)}\triangle
B_j^{(1)}-\xi_j^{(2)}\triangle
B_j^{(2)}))^2]+\E[(\sum_{j=0}^{N-1}(\xi_j^{(1)}\triangle B_j^{(2)}
+\xi_j^{(2)}\triangle B_j^{(1)}))^2]\\
&\leq & 2[\E(\sum_{j=0}^{N-1}(\xi_j^{(1)}\triangle
B_j^{(1)}))^2+\E(\sum_{j=0}^{N-1}(\xi_j^{(2)}\triangle
B_j^{(2)}))^2+\E(\sum_{j=0}^{N-1}(\xi_j^{(1)}\triangle
B_j^{(2)}))^2\\
&+&
\E(\sum_{j=0}^{N-1}(\xi_j^{(2)}\triangle B_j^{(1)}))^2]
\end{eqnarray*}

Also, we have
\begin{eqnarray*}
\E[\xi_j^{(i)}\triangle B_j^{(k)} \xi_{l}^{(m)}\triangle
B_{l}^{(n)}] &=& \E[\xi_j^{(i)}\triangle B_j^{(k)}
\xi_{l}^{(m)}\E[\triangle B_{l}^{(n)}|\Omega_{t_{l}}]]\\
&=& 0,
\end{eqnarray*}
where $l\geq j,and i,k,m,n=1,2$,\\
and
$$
\int_0^T|\eta_t|^2
dt=\sum_{j=0}^{N-1}((\xi_j^{(1)})^2+(\xi_j^{(2)})^2)(t_{j+1}-t_{j}).
$$

We only need to show
$$
\E[[\sum_{j=0}^{N-1}(\xi_j^{(1)}\triangle B_j^{(1)})]^2]\leq
2K\E[\sum_{j=0}^{N-1}((\xi_j^{(1)})^2+(\xi_j^{(2)})^2)(t_{j+1}-t_{j})],
$$
and the others are similar.\\
Notice that
\begin{eqnarray*}
\E[(\xi_j^{(1)})^2(\triangle
B_j^{(1)})^2+(\xi_{j+1}^{(1)})^2(\triangle B_{j+1}^{(1)})^2] &=&
\E[(\xi_j^{(1)})^2(\triangle B_j^{(1)})^2 +
(\xi_{j+1}^{(1)})^2\E[(\triangle B_{j+1}^{(1)})^2|\Omega_{t_j+1}]]\\
&=& \E[(\xi_j^{(1)})^2(\triangle B_j^{(1)})^2 + \bar{\sigma}_1^2
(\xi_{j+1}^{(1)})^2 (t_{j+2}-t_{j+1})]\\
&=& \E[\E[(\xi_j^{(1)})^2(\triangle B_j^{(1)})^2 + \bar{\sigma}_1^2
(\xi_{j+1}^{(1)})^2
(t_{j+2}-t_{j+1})]|\Omega_{t_j}]\\
&\leq& \E[\E[(\xi_j^{(1)})^2(\triangle B_j^{(1)})^2|\Omega_{t_j}] +
\E[ \bar{\sigma}_1^2 (\xi_{j+1}^{(1)})^2
(t_{j+2}-t_{j+1})|\Omega_{t_j}]]\\
&=& \E[(\xi_j^{(1)})^2 \bar{\sigma}_1^2 (t_{j+1}-t_{j})+ \E[
\bar{\sigma}_1^2 (\xi_{j+1}^{(1)})^2
(t_{j+2}-t_{j+1})|\Omega_{t_j}]]\\
&=& \E[\bar{\sigma}_1^2 (\xi_j^{(1)})^2
(t_{j+1}-t_{j})+\bar{\sigma}_1^2 (\xi_{j+1}^{(1)})^2
(t_{j+2}-t_{j+1})].
\end{eqnarray*}
Then we have
\begin{eqnarray*}
\E[[\sum_{j=0}^{N-1}(\xi_j^{(1)}\triangle B_j^{(1)})]^2] &\leq&
2\E[\sum_{j=0}^{N-1}(\xi_j^{(1)})^2(\triangle
B_j^{(1)})^2]+2\E[\sum_{j=0}^{N-1}(\xi_j^{(2)})^2(\triangle
B_j^{(2)})^2]\\
                                                         &\leq&
2\bar{\sigma}_1^2\E[\sum_{j=0}^{N-1}[((\xi_j^{(1)})^2+(\xi_j^{(2)})^2)(t_{j+1}-t_{j})]]\\
&+&
2\bar{\sigma}_2^2\E[\sum_{j=0}^{N-1}[((\xi_j^{(1)})^2+(\xi_j^{(2)})^2)(t_{j+1}-t_{j})]]\\
                                                         &\leq&
2K\E[\sum_{j=0}^{N-1}[((\xi_j^{(1)})^2+(\xi_j^{(2)})^2)(t_{j+1}-t_{j})]]
\end{eqnarray*}

\end{proof}

So we can extend the stochastic integral mapping I from $
M_G^{2,0}(0,T)_{\C}$ to $ M_G^{2}(0,T)_{\C}$.
\begin{def1}
For a $\eta \in  M_G^{2}(0,T)_{\C}$, the stochastic integral is
defined as
\begin{eqnarray*}
\int_0^T\eta_t dB_t&:=& I(\eta),\\
 \int_s^t \eta_u dB_u&:=& \int_0^T I_{[s,t]}(u)
\eta_u dB_u,
\end{eqnarray*}
where I is the extension of the former integral mapping from $
M_G^{2}(0,T)_{\C}$ to $L_G^2(\Omega_T)_{\C}$ by Hahn-Banach
extension theorem.

\end{def1}

Naturally, we have the following basic properties. The proof is
trivial.

\begin{prop}
For $\eta, \zeta \in M_G^{2}(0,T)_{\C}$, and $0\leq s\leq r \leq
t\leq T$, we have:\\
$(i)$ $\int_s^t \eta_u dB_u=\int_s^r \eta_u dB_u+\int_r^t \eta_u
dB_u$.\\
$(ii)$ $\int_s^t(\alpha \eta_u+\zeta_u)dB_u=\alpha \int_s^t \eta_u
dB_u+\int_s^t \zeta_u dB_u$, where $\alpha$ is bounded in
$L_G^2(\Omega_s)_{\C}$.\\
$(iii)$ $\EC[Z+\int_r^T \eta_u dB_u|\Omega_s]=\EC[Z|\Omega_s]$, for
$Z \in L_G^2(\Omega_T)_{\C}$

\end{prop}

Now we turn to Quadratic Variation Process of complex G-Brownian
Motion.\\
For a complex G-B.M. $B_t$, notice
$$
B_t^2=\sum_{j=0}^{N-1}
2B_{t_j^{N}}(B_{t_{j+1}^{N}}-B_{t_j^{N}})+\sum_{j=0}^{N-1}(B_{t_{j+1}^{N}}-B_{t_j^{N}})^2
$$
for a partition $\Pi_N$ of $[0,t]$: $0=t_0\leq t_1\leq,\cdots, \leq
t_N=t$. As $||\Pi_N||\rightarrow 0$, by independence under complex
framework, we can show that
$$
\sum_{j=0}^{N-1}2B_{t_j^{N}}(B_{t_{j+1}^{N}}-B_{t_j^{N}})
\stackrel{L_G^2(\Omega)_{\C}}{\rightarrow} 2\int_0^t B_u dB_u,
$$
so by the completeness of $L_G^2(\Omega)_{\C}$,
$\sum_{j=0}^{N-1}(B_{t_{j+1}^{N}}-B_{t_j^{N}})^2$ must also converge
in $L_G^2(\Omega)_{\C}$.

\begin{def1}
By the argument above, we define
\begin{eqnarray*}
\langle B \rangle_t&:=&\lim_{||\Pi_N||\rightarrow
0}\sum_{j=0}^{N-1}(B_{t_{j+1}^{N}}-B_{t_j^{N}})^2\\
     &=& B_t^2-2\int_0^t B_r dB_r,
     \end{eqnarray*}
and call $\langle B \rangle$ the quadratic variation process of
complex G-Brownian Motion.

\end{def1}

\begin{rem}
Since $L_G^p(\Omega_T)_{\C}=\{\xi_1+i\xi_2|\ \xi_1,\xi_2\in
L_G^p(\Omega_T)\}$, we can similarly show that
\begin{eqnarray*}
\int_0^T \eta_t dB_t &=&
\int_0^T(\eta_t^{(1)}+i\eta_t^{(2)})d(B_t^{(1)}+iB_t^{(2)})\\
                     &=&
[\int_0^T \eta_t^{(1)} dB_t^{(1)}-\int_0^T \eta_t^{(2)} dB_t^{(2)}]
+ i[\int_0^T \eta_t^{(1)} dB_t^{(2)}+\int_0^T \eta_t^{(2)}
dB_t^{(1)}]
\end{eqnarray*}

Furthermore, an algebraic calculation tells:
\begin{eqnarray*}
\langle B \rangle_t &=&
(B_t^{(1)}+iB_t^{(2)})^2-2\int_0^t(B_s^{(1)}+iB_s^{(2)})d(B_s^{(1)}+iB_s^{(2)})\\
      &=&
(B_t^{(1)})^2-(B_t^{(2)})^2+2iB_t^{(1)}B_t^{(2)}\\
      &-&
 2[\int_0^t
B_s^{(1)} dB_s^{(1)}-\int_0^t B_s^{(2)} dB_s^{(2)}+i\int_0^t
B_s^{(1)} dB_s^{(2)}+
i\int_0^t B_s^{(2)} dB_s^{(1)}]\\
      &=&
[(B_t^{(1)})^2-2\int_0^t B_s^{(1)}
dB_s^{(1)}]-[(B_t^{(2)})^2-2\int_0^t B_s^{(2)} dB_s^{(2)}]\\
      &+&
2i[B_t^{(1)}B_t^{(2)}-\int_0^t B_s^{(1)} dB_s^{(2)}-\int_0^t
B_s^{(2)} dB_s^{(1)}]\\
      &=&
\langle B^{(1)} \rangle _t-\langle B^{(2)}\rangle _t+2i\langle
B^{(1)},B^{(2)}\rangle_t
\end{eqnarray*}

\end{rem}

Here are some basic properties of $\langle B \rangle$. For
simplicity, we denote
\begin{eqnarray*}
\overrightarrow{B_t}&=&(B^{(1)}_t,B^{(2)}_t)\\
&&\\
\overrightarrow{{\langle B \rangle}}_t&=& \left(\begin{array}{ccc}
\langle B^{(1)} \rangle _t & \langle B^{(1)},B^{(2)}\rangle_t\\
&&\\
\langle B^{(1)},B^{(2)}\rangle_t  & \langle B^{(2)} \rangle _t
\end{array}\right)
\end{eqnarray*}

\begin{rem}
Since $\overrightarrow{B_t}$ is a two dimensional G-Brownian Motion,
$\overrightarrow{{\langle B \rangle}}_t$ is maximal distributed, so
are $Re{\langle B \rangle}_t$ and $Im{\langle B \rangle}_t$. Then we
can similarly define maximal distribution under complex framework in
a trivial way. Then, of course, ${\langle B \rangle}_t$ is maximal
distributed.
\end{rem}

%

\begin{lemma}
For $s,t \geq 0$, ${\langle B \rangle}_{t+s}-{\langle B \rangle}_t$
is identically distributed with ${\langle B \rangle}_s$ and
independent from $\Omega_s$.
\end{lemma}

\begin{proof}
Notice $\langle B \rangle_t={\langle B^{(1)} \rangle}_{t}-{\langle
B^{(2)} \rangle}_{t}+2i\langle B^{(1)},B^{(2)}\rangle_{t}$, and by
the definition of independence we have the conclusion.
\end{proof}

Then we need to define integral with respect to $\langle B
\rangle_t$. Firstly, we define the map $J$ from
$M_G^{1,0}(0,T)_{\C}$ to $L_G^1(\Omega_T)_{\C}$:
$$
J(\eta)=\int_0^T \eta_td\langle B
\rangle_t:=\sum_{j=0}^{N-1}\xi_j(\langle B \rangle_{t_{j+1}}-\langle
B \rangle_{t_{j}}),
$$
where
$\eta_t(\omega)=\Sigma_{k=0}^{N-1}\xi_k(\omega)I_{[t_k,t_k+1)}(t)$
and then by the following lemma, we could have $\int_0^T \eta_t
d\langle B \rangle_t$ for $\eta_t \in M_G^{1}(0,T)_{\C}.$

\begin{lemma}
For $\eta_t=\Sigma_{k=0}^{N-1}\xi_k I_{[t_k,t_k+1)}(t) \in
M_G^{1}(0,T)_{\C},$
$$
\E[|J(\eta)|] \leq
4(\bar{\sigma}_1^2+\bar{\sigma}_3^2+2\bar{\sigma}_2^2)\E[\int_0^T
|\eta_t|dt],
$$
where $\bar{\sigma}_i^2=\sup_{\Lambda \in \Sigma}
|{\sigma}_i^2|,i=1,2,3,$ and $G(A)=\frac{1}{2}\sup_{\Lambda \in
\Sigma} (A,\Lambda)$, with $\Sigma$ a bounded closed convex subset
of $2\times2$ symmetric matrix and $\Lambda=\left(\begin{array}{cc}
\sigma_1^2&\sigma_2^2\\
\sigma_2^2&\sigma_3^2
\end{array}\right)$
\end{lemma}

\begin{proof}
We denote
\begin{eqnarray*}
\triangle t_j&=&t_{j+1}-t_j\\
\triangle \langle B \rangle_j &=&
\langle B\rangle_{t_{j+1}}-\langle B \rangle_{t_{j}},\\
\triangle \langle B^{(i)} \rangle_j &=& \langle B^{(i)}
\rangle_{t_{j+1}}-\langle B^{(i)} \rangle_{t_{j}},i=1,2,\\
\triangle \langle B^{(1)},B^{(2)} \rangle_j &=& \langle
B^{(1)},B^{(2)} \rangle_{t_{j+1}}-\langle B^{(1)},B^{(2)}
\rangle_{t_{j}}.
\end{eqnarray*}

We have $\E[|\xi_j \triangle \langle B\rangle_j|] \leq \E[|Re(\xi_j
\triangle \langle B\rangle_j)|]+\E[|Im(\xi_j \triangle \langle
B\rangle_j)|]$ and then
\begin{eqnarray*}
&& \E[|Re(\xi_j \triangle \langle B\rangle_j)|]\\ &=&
\E[|\xi_j^{(1)} \triangle \langle B^{(1)} \rangle_j - \xi_j^{(1)}
\triangle \langle B^{(2)} \rangle_j -2 \xi_j^{(2)} \triangle \langle
B^{(1)}, B^{(2)} \rangle_j|]\\ & \leq &
(\bar{\sigma}_1^2+\bar{\sigma}_3^2+2\bar{\sigma}_2^2)\E[\triangle
t_j (|\xi_j^{(1)}|+|\xi_j^{(2)}|)].
\end{eqnarray*}
The similar result holds for $Im(\xi_j \triangle \langle
B\rangle_j),$ so we have
$$
\E[|\sum_{j=0}^{N-1} \xi_j \triangle \langle B\rangle_j|] \leq 2
(\bar{\sigma}_1^2+\bar{\sigma}_3^2+2\bar{\sigma}_2^2)\E[\triangle
t_j (|\xi_j^{(1)}|+|\xi_j^{(2)}|)].
$$
The inequality $|a|+|b| \leq 2\sqrt{a^2+b^2}$ would finish the
proof.
\end{proof}

\begin{prop}
Here are some properties for the integral with respect to $\langle
B\rangle.$\\
$(i)$ For $0 \leq s \leq t ,$ $\xi \in L_G^2(\Omega_s)_{\C},Z \in
L_G^1(\Omega_s)_{\C},$ we have
$$
\EC[Z+\xi (B_t^2-B_s^2)]=\EC[Z+\xi (\langle B\rangle_t-\langle
B\rangle_s)]=\EC[Z+\xi(B_t-B_s)^2]
$$
$(ii)$ For $\eta \in M_G^{2,0}(0,T)_{\C}$, we have
$$
\EC[(\int_0^T \eta_t dB_t)^2]=\EC[\int_0^T \eta_t^2 d\langle
B\rangle_t]
$$
\end{prop}
\begin{proof}
For $(i)$, by the definition of $\langle B\rangle_t$ and proposition
(\ref{10}), we have
\begin{eqnarray*}
\EC[Z+\xi (\langle B\rangle_t-\langle
B\rangle_s)]&=&\EC[Z+\xi(B_t^2-B_s^2+\int_s^t B_r dB_r)]\\
            &=&\EC[Z+\xi(B_t^2-B_s^2)]\\
            &=&\EC[Z+\xi((B_t-B_s)^2+2B_s(B_t-B_s))]\\
            &=&\EC[Z+\xi(B_t-B_s)^2]
\end{eqnarray*}
For $(ii)$, notice
$\EC[Z+2\xi_j(B_{t_{j+1}}-B_{t_j})\xi_i(B_{t_{i+1}}-B_{t_i})]=\EC[Z]$
and then
$$
\EC[(\int_0^T \eta_s dB_s)^2]=\EC[(\sum_{j=0}^{N-1}\xi_j
(B_{t_{j+1}}-B_{t_j}))^2].
$$
By $(i)$, we would have the conclusion.

\end{proof}

\section{Complex $It\hat{o}$ Formula and Conformal G-Brownian Motion}

Now we turn to $It\hat{o}$ formula under this framework. Firstly we
need to set some basic definitions.

\begin{def1}\label{partial}
For $f=u+iv$ differentiable with real part $u$ and imaginary part
$v$, we define operator $\partial f$ and $\bar{\partial}f$ as:
$$
\partial f:=\frac{\partial}{\partial z}f:=\frac12 (\frac{\partial f}{\partial x}-i\frac{\partial f}{\partial
y})=\frac12(\frac{\partial u}{\partial x}+\frac{\partial v}{\partial
y}+i(\frac{\partial v}{\partial x}-\frac{\partial u}{\partial y}))
$$

$$
\bar{\partial}f:=\frac{\partial}{\partial \bar{z}}f:=\frac12
(\frac{\partial f}{\partial x}+i\frac{\partial f}{\partial
y})=\frac12(\frac{\partial u}{\partial x}-\frac{\partial v}{\partial
y}+i(\frac{\partial v}{\partial x}+\frac{\partial u}{\partial y}))
$$
\end{def1}

\begin{rem}
For $\hat{B}_i=\alpha_i B_t^{(1)}+i \beta_i B_t^{(2)}$, $i=1,2$,
where $\alpha_i , \beta_i \in \R$ and $B_t^{(1)}+i B_t^{(2)}$ is a
complex G-Brownian Motion, we have $(\alpha_i B_t^{(1)}, \beta_i
B_t^{(2)})$ is also a two-dimensional $\tilde{G}_i$-Brownian Motion,
so $\hat{B}_i$ is also a complex $\tilde{G}_i$-Brownian Motion. Then
we can define
$$
\langle \hat{B}_1,\hat{B}_2 \rangle_t:=\frac{1}{4}( \langle
\hat{B}_1+ \hat{B}_2 \rangle_t - \langle \hat{B}_1- \hat{B}_2
\rangle_t).
$$
and we have

\begin{eqnarray*}
&& \langle \hat{B}_1,\hat{B}_2 \rangle_t\\
&=& \langle \alpha_1 B^{(1)}+i \beta_1 B^{(2)},\alpha_2
B^{(1)}+i \beta_2 B^{(2)} \rangle_t\\
&=& \frac14[\langle
(\alpha_1+\alpha_2)B^{(1)}+i(\beta_1+\beta_2)B^{(2)}
\rangle_t-\langle
(\alpha_1-\alpha_2)B^{(1)}+i(\beta_1-\beta_2)B^{(2)} \rangle_t
]\\
&=& \frac14[\langle (\alpha_1+\alpha_2)B^{(1)} \rangle_t-\langle
(\beta_1+\beta_2)B^{(2)} \rangle_t+2i\langle
(\alpha_1+\alpha_2)B^{(1)},(\beta_1+\beta_2)B^{(2)} \rangle_t\\
&-& \langle (\alpha_1-\alpha_2)B^{(1)} \rangle_t+\langle
(\beta_1-\beta_2)B^{(2)} \rangle_t-2i\langle
(\alpha_1-\alpha_2)B^{(1)},(\beta_1-\beta_2)B^{(2)} \rangle_t]\\
&=& \alpha_1 \alpha_2 \langle B^{(1)} \rangle_t - \beta_1
\beta_2\langle B^{(2)} \rangle_t +i (\alpha_1 \beta_2 +\beta_1
\alpha_2) \langle B^{(1)},B^{(2)} \rangle_t.
\end{eqnarray*}
In particular, $\langle B, \bar{B} \rangle_t=\langle
B^{(1)}\rangle_t + \langle B^{(2)}\rangle_t.$
\end{rem}
Furthermore, $\int_0^T \eta_s \langle B_1,B_2 \rangle_s $ also
explains itself.

\begin{thm}(Complex Version of $It\hat{o}'s$ Lemma)\label{Cito} For
$$
Z_t=Z_0+\int_0^t \alpha_s ds +\int_0^t \eta_s d\langle B
\rangle_s+\int_0^t \beta_s dB_s,
$$
where $\alpha, \beta, \eta $ are all bounded processes in
$M_G^2(0,T)_{\C},$ and any function $f$, which is twice continuously
differentiable and satisfies polynomial growth condition for the
second order derivatives, we have the following equation in
$L_G^2(\Omega_t)_{\C}$, $\forall t \geq 0  :$
\begin{eqnarray*}
f(Z_t)-f(Z_0)
&=& \int_0^t \partial f(Z_s) dZ_s + \int_0^t
\bar{\partial} f(Z_s) d\bar{Z}_s+\int_0^t \partial \partial f(Z_s)
d\langle Z \rangle_s \\
&+& \int_0^t \bar{\partial} \bar{\partial}
f(Z_s) d\langle \bar{Z} \rangle_s +2 \int_0^t {\partial}
\bar{\partial}
f(Z_s) d\langle Z,\bar{Z} \rangle_s\\
&=& \int_0^t (\partial f \alpha_s + \bar{\partial} f
\bar{\alpha}_s)ds +\int_0^t \partial f \beta_s dB_s + \int_0^t
\bar{\partial} f \bar{\beta}_s d\bar{B}_s+\int_0^t \partial f \eta_s
d \langle B\rangle_s + \int_0^t \bar{\partial} f \bar{\eta}_s d
\overline{\langle
B\rangle}_s\\
&+& \int_0^t \partial \partial f \beta_s^2 d \langle B \rangle_s +
\int_0^t \bar{\partial} \bar{\partial} f \bar{\beta}_s^2 d \langle
\bar{B} \rangle_s +2\int_0^t \partial \bar{\partial}f {|\beta_s|}^2
d \langle B, \bar{B}\rangle_s
\end{eqnarray*}
\end{thm}

\begin{proof}
Notice $\overline{\langle B\rangle}_s=\langle B^{(1)} \rangle_s-
\langle B^{(2)} \rangle_s-2 i \langle B^{(1)}, B^{(2)} \rangle_s.$
The proof can be done by a review of $it\hat{o}$ formula in the real
case and a careful algebraic calculation.
\end{proof}
\begin{rem}
The first equation in the above theorem is formal to avoid the
complex structure of the second equation. See \cite{Revuz},
\cite{conformal}.

\end{rem}
\begin{rem}
If f is twice continuously differentiable, we have
$\frac{\partial}{\partial x} \frac{\partial}{\partial
y}f=\frac{\partial}{\partial y} \frac{\partial}{\partial x}f,$ and
further $\partial \bar{\partial}f= \bar{\partial}\partial f=\frac14
\Delta f$, where $\Delta $ is the laplace operator, i.e.
$\Delta:=(\frac{\partial}{\partial x})^2+(\frac{\partial}{\partial
y})^2.$
\end{rem}

We know that a complex G-Brownian Motion can be viewed as a
two-dimensional real G-Brownian Motion, which includes too many
elements to get better properties. Here is a special kind of complex
G-Brownian Motion, called conformal G-Brownian Motion, which is the
main object of complex stochastic analysis.

\begin{def1}
A complex G-Brownian Motion is called a conformal G-Brownian Motion
if $\langle B \rangle_t=0$ in $L_G^2(\Omega_t)_{\C}$ for any $t \geq
0.$
\end{def1}

\begin{rem}\label{conformal}
Notice $\langle B \rangle_t=\langle B^{(1)} \rangle_t-\langle
B^{(2)} \rangle_t+2i \langle B^{(1)},B^{(2)} \rangle_t,$ so $\langle
B \rangle_t\equiv0$ if and only if $\langle B^{(1)}
\rangle_t=\langle B^{(2)} \rangle_t$ and $\langle B^{(1)},B^{(2)}
\rangle_t=0.$ This means the real part and the imaginary part moves
as the same rate (identically distributed) and they are irrelevant.
In the classical case, a complex Brownian Motion is surely
conformal, since its real part and imaginary part are independent.
However, things are a little different under G-framework. We cannot
say $B^{(1)}$ and $B^{(2)}$ are independent under G-framework.
\end{rem}

\begin{example}
For a random vector $X=(X_1,X_2)$, where $X_1$ is a real G-normal
distributed variable with $\bar{\sigma}^2> \underline{\sigma}^2$,
and $X_2$ is an independent copy of $X_1$, we can claim that $X$
fails to be a real G-normal distributed vector, which can be easily
checked by the definition of real G-normal distribution. In fact, if
$\bar{X}=(\bar{X}_1,\bar{X}_2)$ is an independent copy of $X,$
$\bar{X}_2$ is independent of $\bar{X}_1$ by the definition of
independence, so for $\varphi(x,y)=x^2y,$ we have
\begin{eqnarray*}
\E[\varphi(X+\bar{X})]&=&\E[(X_1+\bar{X}_1)^2(X_2+\bar{X}_2)]\\
                      &=&\E[X_2^+\bar{\sigma}^2-X_2^{-}\underline{\sigma}^2]\\
                      &=&(\bar{\sigma}^2-\underline{\sigma}^2)\E[X_2^+]\\
                      &=&\frac12 (\bar{\sigma}^2-\underline{\sigma}^2) \E[|X_2|]>0
\end{eqnarray*}
while $\E[\varphi(\sqrt{2}X)]=0.$\\
This means a nontrivial two dimensional G-normal distributed vector
fails to have independent elements, so does a complex G-normal
distributed variable.
\end{example}

Before giving a description of complex G-conformal Brownian Motion,
we need a simple fact.

\begin{lemma}\label{0qs}
If $X$ is a real maximal distributed n-dimensional vector and
satisfies
$$
\E[\varphi(X)]=\varphi(0)
$$
for any $\varphi \in C_{l.lip}(\R^n).$ Then we have $X=0,$ $q.s$
\end{lemma}

\begin{proof}
Take $\varphi(x)=|x|^2.$ Since $X$ is maximal distributed,
$$
\E[\varphi(X)]=\sup_{(X_1,\cdots,X_n) \in V}(X_1^2+\cdots+X_n^2)=0
$$
where $V$ is a convex closed subset of $\R^n$ (see
remark(\ref{md})). It must be $V=\{0\},$ and we have the conclusion.

\end{proof}

Here is a description of conformal G-Brownian Motion.

\begin{thm}
A complex G-Brownian Motion $B_t$ is conformal if and only if
$G(\cdot)$ has the following expression :
$$
G(A):=\E[\overrightarrow{B_t}A\overrightarrow{B_t}^T]=\frac12
\sup_{\Lambda \in \Sigma}(A,\Lambda)
$$
where $\Sigma=\{\left(\begin{array}{cc}
\sigma^2 & 0\\
0 & \sigma^2
\end{array}\right) | \sigma^2 \in
[\underline{\sigma}^2,\bar{\sigma}^2]\},$ and  $\bar{\sigma}^2 \geq
\underline{\sigma}^2  \geq 0.$

\end{thm}

\begin{proof}

According to remark (\ref{conformal}), the only if part is simple.
In fact, if we denote
$$\Lambda=\left(\begin{array}{cc}
\sigma_1^2 & \sigma_2^2\\
\sigma_2^2 & \sigma_3^2
\end{array}\right),
$$
since $\langle B^{(1)} \rangle_t=\langle B^{(2)} \rangle_t$ and
$\langle B^{(1)},B^{(2)} \rangle_t=0,$ we have
$$
\E[f(\langle B^{(1)} \rangle_t-\langle B^{(2)}
\rangle_t)]=\sup_{\Lambda \in \Sigma}f(\sigma_1^2-\sigma_3^2)=f(0)
$$
and
$$
\E[g(\langle B^{(1)},B^{(2)} \rangle_t)]=\sup_{\Lambda \in
\Sigma}g(\sigma_2^2)=g(0),
$$
for any $f,g \in C_{l.lip}(\R).$ It follows that $\Sigma$ must have
the above expression.

For the if part, notice $B_t$ is uniquely determined by $G(A),$ and
$$
\E[\varphi(\overrightarrow{{\langle B \rangle}}_t)]=\sup_{\Lambda
\in \Sigma}\varphi(\Lambda).
$$
We have
$$
\E[\varphi(\langle B^{(1)} \rangle_t-\langle B^{(2)}
\rangle_t)]=\sup_{\Lambda \in \Sigma}
\varphi(\sigma^2-\sigma^2)=\varphi(0)
$$
and
$$
\E[\varphi(\langle B^{(1)}, B^{(2)} \rangle_t)]=\sup_{ \Lambda \in
\Sigma} \varphi(0)= \varphi(0),
$$
for any $\varphi \in C_{l.lip}(\R).$ By lemma (\ref{0qs}), we get
the conclusion.
\end{proof}

We give the definition of G-martingale under complex case in a
trivial way.
\begin{def1}
For a complex process $(M_t)_{t \geq 0},$ it is called a complex
G-martingale if $M_t \in L_G^1(\Omega_t)_{\C},$ and
$$
\EC[M_t|\Omega_s]=M_s,
$$
for $0 \leq s \leq t.$
\end{def1}

Here is a property of analytic function, which we will use in the
next.

\begin{lemma}
$f$ is continuously differentiable complex function. Then $f$ is
analytic if and only if $\partial f=0$. In this case,
$$
f'(z)=\partial f(z),
$$
here $f'(z)$ means the derivative of $f$ in the complex sense.
\end{lemma}

\begin{proof}
Notice that $\partial f=0$ if and only if $\frac{\partial
u}{\partial x}=\frac{\partial v}{\partial y},\frac{\partial
u}{\partial y}=-\frac{\partial v}{\partial x},$ when $f=u+iv.$ The
conclusion follows from Cauchy-Riemann equation.

\end{proof}

\begin{corollary}
Suppose $B_t$ is a conformal G-Brownian Motion, and $f$ is analytic.
Then $f(B_t)$ is a symmetric complex G-martingale.
\end{corollary}

\begin{proof}
In this case, by $it\hat{o}'s$ formula and the above lemma, we have
$$
f(B_t)-f(B_0)=\int_0^t f'(B_s)dB_s.
$$

\end{proof}

\begin{example}
Suppose $B_t$ is conformal. Then $B_t^2$ is a martingale. In face, $
B_t^2=(B_t^{(1)})^2-(B_t^{(2)})^2+2iB_t^{(1)}B_t^{(2)} ,$ so we have
\begin{eqnarray*}
\EC[B_t^2|\Omega_s]&=&\E[(B_t^{(1)})^2-(B_t^{(2)})^2|\Omega_s]+2i\E[B_t^{(1)}B_t^{(2)}|\Omega_s]\\
&=&\E[(B_t^{(1)})^2-\langle B^{(1)}\rangle_t-((B_t^{(2)})^2-\langle
B^{(2)}\rangle_t)+\langle B^{(1)}\rangle_t-\langle
B^{(2)}\rangle_t|\Omega_s]\\
&+& 2iB_s^{(1)} B_s^{(2)}\\
&=&(B_s^{(1)})^2-(B_s^{(2)})^2+2iB_s^{(1)} B_s^{(2)}\\
&=&B_s^2
\end{eqnarray*}

\end{example}

In fact, we can take this conclusion further to get the conformal
invariance by considering martingale with the form $M_t=\int_0^t
\eta_u dB_u,$ where $\eta_u \in M_G^2(0,T)_{\C}.$ Since $M_t$ is a
symmetric martingale, we can define the quadratic variation in the
old fashion way: the limit point under norm of
$L_G^2(\Omega_T)_{\C},$ and then we would have
$$
\langle M \rangle_t=\int_0^t \eta_s^2 d\langle B \rangle_s.
$$

\begin{def1}
A complex martingale $M_t$ is conformal if $\langle M \rangle_t=0.$
\end{def1}

Obviously, a conformal G-Brownian motion is a conformal martingale.
Furthermore, we suppose $\eta_u \in M_G^4(0,T)_{\C}.$ If $B_t$ is
conformal, by $it\hat{o}'$s lemma and B-D-G inequity, we have
$M_t^2=2 \int_0^2 M_s \eta_s dB_s ,$ so $\langle M^2 \rangle_t=4
\int_0^t M_s^2 \eta_s^2 d\langle B \rangle_s.$ Notice $d\langle M
\rangle_s=\eta_s^2 d\langle B \rangle_s,$ and we get $\langle M^2
\rangle_t=4 \int_0^t M_s^2 d\langle M \rangle_s.$ In conclusion, we
get:

\begin{corollary}
$B_t$ is conformal and $M_t$ is defined as above. Then $M_t,M_t^2$
are conformal martingales.
\end{corollary}

\begin{corollary}
If $M_t=\int_0^t \eta_s dB_s$ is conformal, with values in an open
set $E,$ and $f$ is a bounded analytic function with bounded first
order derivative on $E$, polynomial growth of the second order
derivative, then $f(M_t)$ is conformal. Furthermore,
$$
\langle f(M),\overline{f(M)} \rangle_t=\int_0^t f'(M_s)
\overline{f'(M_s)}d\langle M,\bar{M} \rangle_s
$$

\end{corollary}

\begin{proof}

By it$\hat{o}'s$ lemma, $f(M_t)=\int_0^t f'(M_s)\eta_s dB_s,$ so
\begin{eqnarray*}
\langle f(M) \rangle_t &=& \int_0^t(f'(M_s))^2 \eta_s^2 d\langle B
\rangle_s\\
&=& \int_0^t (f'(M_s))^2 d\langle M \rangle_s\\
&=&0
\end{eqnarray*}

The equation follows from the fact that $f(M)$ is also a symmetric
martingale.

\end{proof}


\renewcommand{\refname}{\large References}{\normalsize \ }

\end{document}